\documentclass[letterpaper, 12pt]{article}
\usepackage{mathrsfs}
\usepackage[colorlinks, citecolor=blue, linkcolor=blue]{hyperref}
\usepackage{xcolor}

\usepackage{mathscinet}
\usepackage{latexsym}
\usepackage{amsthm}
\usepackage{amssymb}
\usepackage{amsfonts}
\usepackage{amsmath}
\usepackage{comment}

\DeclareMathOperator{\RE}{Re}
\DeclareMathOperator{\IM}{Im}
\def\Enum{\mathbb{E}}
\def\D{\mathcal{D}}

\newcommand{\Rnum}{\mathbb{R}}
\newcommand{\Cnum}{\mathbb{C}}
\newcommand{\abs}[1]{\left\vert#1\right\vert}
\newcommand{\innp}[1]{\langle {#1}\rangle}
\newcommand{\FH}{\mathfrak{H}}
\newcommand{\E}{\mathbb{E}}
\newcommand{\set}[1]{\left\{#1\right\}}
\newcommand{\dif}{\mathrm{d}}
\newcommand{\mi}{\mathrm{i}}
\newcommand{\norm}[1]{\left\Vert#1\right\Vert}
\newcommand\sgn{\mathrm{sgn}}
\newcommand{\N}{\mathbb{N}}

\theoremstyle{plain}
\newtheorem{theorem}{Theorem}[section]
\newtheorem{lemma}[theorem]{Lemma}
\newtheorem{corollary}[theorem]{Corollary}
\newtheorem{proposition}[theorem]{Proposition}

\theoremstyle{definition}
\newtheorem{definition}[theorem]{Definition}

\theoremstyle{remark}
\newtheorem{remark}{Remark}

\begin{document}

\title{\Large Parameter Estimation for Complex $\alpha$-Fractional Brownian Bridge}

\date{}
\author{
  Yong Chen\footnote{School of Big Data, Baoshan University, Baoshan, Yunnan, China}
  \and
  Lin Fang\footnote{School of Mathematics and Statistics, Jiangxi Normal University, Nanchang, Jiangxi, China}
  \and
  Ying Li\footnote{School of Mathematics and Computational Science, Xiangtan University, Xiangtan, Hunan, China}
  \and
  Hongjuan Zhou\textsuperscript{*}\footnote{\textsuperscript{*}Corresponding author. Email: Hongjuan.Zhou@asu.edu. School of Mathematical and Statistical Sciences, Arizona State University, Tempe, Arizona, USA. }
}

\maketitle

\textbf{Abstract:}
We study the statistical inference problem for a complex $\alpha$-fractional Brownian bridge process $Z$ defined by the stochastic differential equation
\[ \mathrm{d}Z_t = -\alpha \frac{Z_t}{T - t} \mathrm{d}t + \mathrm{d}\zeta_t,  \quad t \in [0,  T), \]
with initial condition $Z_0 = 0$, where $\alpha = \lambda - \sqrt{-1}w$, $\lambda > 0$, $w \in \mathbb{R}$ and $\zeta_t$ is a complex fractional Brownian motion. We establish the well-posedness of the fractional Brownian bridge $Z_t$ over the time interval $[0, T]$ for all $H \in (0, 1)$, and prove the strong consistency and the asymptotic distribution for the classic least squares estimator of the parameter \(\alpha\) when \(H \in \left(\frac{1}{2},  1\right)\). The proofs are based on stochastic analysis elements about complex multiple Wiener-It\^o integrals and the complex Malliavin calculus. Unlike the real-valued fractional Brownian bridge considered in the literature, the two-dimensional limiting distribution has non-Cauchy marginal distributions.

\vskip 4mm

\section{Introduction}

Statistical inference for one-dimensional stochastic differential equations (SDE) driven by fractional Brownian motion (fBm) has been extensively investigated in the literature. In contrast, statistical estimation for multi-dimensional fractional stochastic equation remains relatively underexplored (see \cite{HNZ} and the references therein). To the best of our knowledge, existing studies in this area have concentrated primarily on the complex Vasicek model, which encompasses the complex Ornstein-Uhlenbeck process as a special case (see  \cite{AACZ}, \cite{CHW}, \cite{GTC}, \cite{STY} and the references therein). The analytical techniques developed in these works are, in principal, applicable to a broad class of stochastic processes driven by fractional Brownian motion.

Brownian bridges arise naturally in the characterization of stochastic fluctuations with fixed boundary conditions and have found applications in diverse fields, including polymer chain models in physics and trait evolution models in biology. A real-valued $\alpha$-fractional Brownian bridge is defined as the pathwise solution to the SDE 
\begin{align}\label{real ALPHA BRIDGE}
   \dif X_t=-\alpha\frac{X_t}{T-t}\dif t+\sigma\dif B_t^H,\,\, t\in[0,T)
\end{align} 
with initial condition $X_0=0$, where $\alpha, \sigma$ are positive real numbers and $B_t^H$ is a fBm with Hurst parameter $H\in (0,1)$. When setting up $\sigma=1$, the unique solution to equation \eqref{real ALPHA BRIDGE} is
\begin{align}\label{xt bds}
    X_t=(T-t)^{\alpha}\int_0^t (T-u)^{-\alpha} \dif B^H_u,\quad \quad 0\le t<  T.
\end{align}
The parameter $\alpha$ controls the strength of the attraction toward the endpoint. When $\alpha=1$, it is the standard Brownian bridge. 

It is interesting to investigate the statistical estimation for the parameter $\alpha$. This statistical problem was first studied in \cite{EN}. The least squares estimator (LSE) for the real-valued $\alpha$-fractional Brownian bridge is given by
 \begin{align}\label{LSE}
     \hat{\alpha}_t=-\frac{\int_0^t \frac{{X}_u}{T-u}\dif X_u}{\int_0^t\frac{X_u^2}{(T-u)^2}\dif u},
 \end{align}
 where the stochastic integral in the numerator is a Young integral when $H\in (\frac12,1)$. Under the condition of $\alpha\le \frac12$, the LSE $\hat{\alpha}_t$ is strongly consistent and the term $\hat{\alpha}_t-\alpha$ after being normalized converges in law as $t\to T$ (see \cite{EN}). Similar results have been found for the SDE \eqref{real ALPHA BRIDGE} when $B^H$ is replaced by some other Gaussian processes (see \cite{HSY} and \cite{KL}). Clearly, when $H\in(0,\frac12)$, the stochastic integral $\int_0^t \frac{{X}_u}{T-u}\dif X_u $ cannot be interpreted as a Young integral any more. 
 Moreover, a similar result is extended to the SDE \eqref{real ALPHA BRIDGE} with linear drift
 (see \cite{HShY}).

This paper aims to contribute to this research strand by studying the least squares estimator for the parameter $\alpha$ of a complex $\alpha$-fractional Brownian bridge.
Let $(B_t^1,B_t^2)$ be a two-dimensional fBm with the same Hurst parameter $H\in (0,1)$. Consider the SDE driven by the complex fractional Brownian motion $\zeta_t=\frac{1}{\sqrt{2}}\left(B_t^1+\mi B_t^2\right)$, 
\begin{align}\label{COMPLEX ALPHA BRIDGE}
   \dif Z_t=-\alpha\frac{Z_t}{T-t}\dif t+\sigma\dif \zeta_t,\quad t\in[0,T)
\end{align} 
with initial condition $Z_0=0$, where $\alpha=\lambda-\mi w$ is a complex number with positive real part and $\sigma>0$. The pathwise solution to this SDE is a complex-valued process, referred to as a complex $\alpha$-fractional Brownian bridge $Z_t=X_1(t)+\mi X_2(t)$.

Suppose that only one trajectory for the stochastic process $Z=(Z_u)_{u\in[0,t)}$ can be obtained. We would like to construct a consistent estimator for the unknown complex parameter $\alpha$ and study its asymptotic behavior as $t\to T$. To compare with the work in \cite{EN}
, we write \eqref{COMPLEX ALPHA BRIDGE} in multi-dimensional SDE form as
    \begin{align}\label{2dim eq}
	\left[\begin{array}{c}
	\dif X_1(t)  \\
	\dif X_2(t)
        \end{array}\right]
			=\frac{1}{T-t}\left[\begin{array}{cc}
				-\lambda &-w  \\
				w &-\lambda
			\end{array}\right]
			\left[\begin{array}{c}
				X_1(t)  \\
				X_2(t) 
			\end{array}\right]\dif t
			+\frac{\sigma}{\sqrt{2}}\left[\begin{array}{c}
				\dif B_t^1  \\
				\dif B_t^2 
			\end{array}\right].
		\end{align}

When the parameter $w$ degenerates to zero, the system of equations \eqref{2dim eq} reduces to two independent real $\lambda$-fractional Brownian bridges (see the equation \eqref{real ALPHA BRIDGE} above), which have been studied in \cite{EN}. When the parameter $w$ dose not equal to zero, the tools used in  \cite{EN} can not be applied to analyze the system \eqref{2dim eq}. Instead, we will apply complex Malliavin calculus.

By setting up $\sigma= 1 $ in the model \eqref{COMPLEX ALPHA BRIDGE} without loss of generality, the solution to equation \eqref{COMPLEX ALPHA BRIDGE} can be written explicitly as 
\begin{align}\label{zt bds}
				Z_t=(T-t)^{\alpha}\int_0^t (T-u)^{-\alpha} \dif \zeta_u,\quad t\in [0,T).
			\end{align}
Denote the complex Wiener integral term in \eqref{zt bds} by
\begin{align}\label{omega bds}
\omega_t=\int_0^t (T-u)^{-\alpha} \dif \zeta_u,\quad t\in [0,T).
\end{align} 
Our first result establishes that the process $Z_t$ defines a well-posed complex $\alpha$-fractional Brownian bridge over $t \in [0, T]$. This is stated precisely in the following theorem.

\begin{theorem}\label{fBr existence}
  Assume $H \in (0, 1)$. Suppose $\lambda = \RE(\alpha) \in (0, H)$. Then the limit $\omega_T:=\lim_{t \uparrow T} \omega_t$ exists in $L^2$ and almost surely. Moreover,
    \begin{align}\label{wt2 norm}
	\Enum [|\omega_T|^2]=\frac{\Gamma(1+2H) }{2(H-\lambda)}\RE \Big[\frac{\Gamma(1-\alpha)}{\Gamma(2H-\alpha)} \Big]T^{2(H-\lambda)}.
    \end{align}
Consequently, the Gaussian process $\omega:=\{\omega_t, t \in [0, T]\}$ admits a modification on $[0,T] $ with $(H-\lambda-\epsilon)$-H\"{o}lder continuous paths for all $\epsilon\in (0,H-\lambda)$. In addition, the process $Z=(Z_t)_{t\in [0,T]}$ defines a well-posed complex $\alpha$-fractional Brownian bridge.
\end{theorem}

Now for the complex $\alpha$-fractional Brownian bridge $Z$, we consider the statistical estimation of the important parameter $\alpha$ when $H\in (\frac12,1)$. We write equation \eqref{COMPLEX ALPHA BRIDGE} intuitively as 
$$\dot{Z}_t+\alpha \frac{Z_t}{T-t}=\dot{\zeta}_t,\,\,\, 0\leq t< T.$$
By minimizing $$F(\alpha): =\int_0^t \left|\dot{Z}_u+\alpha \frac{Z_u}{T-u}\right|^2\dif  u,$$ 
we obtain the equation 
$$\frac{\partial F}{{\partial\bar{\alpha}}}=\int_0^t \frac{\bar{Z}_u}{T-u} \dif Z_u + \alpha \int_0^t \frac{|\bar{Z}_u|^2}{(T-u)^2} \dif u=0.$$
Solving the above equation yields the least squares estimator of the complex parameter $\alpha$:  \begin{align}\label{ALPHA LSE}
			\hat{\alpha}_t=-\frac{\int_0^t \frac{\overline{Z}_u}{T-u}\dif Z_u}{\int_0^t\frac{|Z_u|^2}{(T-u)^2}\dif u}
			=\alpha-\frac{\int_0^t \frac{\overline{Z}_u}{T-u}\dif \zeta_u}{\int_0^t\frac{|Z_u|^2}{(T-u)^2}\dif u}.
		\end{align}
The stochastic integral in the numerator is a complex Young integral since $H\in (\frac12,1)$. 

It is important to study the asymptotic behavior of the estimator $\hat\alpha_t$. We write $\xrightarrow{a.s.}$ and $\xrightarrow{d}$ for convergence almost surely and in distribution as $t \to T$, respectively. The second contribution of this paper is to show that the LSE $\hat{\alpha}_t$ is consistent, which is presented in Theorem~\ref{main result}.
\begin{theorem}\label{main result}
    Let $H\in (\frac12, 1)$. When $\RE(\alpha) \in (0, \frac{1}{2}] $, we have $\hat{\alpha}_t \xrightarrow{a.s.} \alpha$ as $t \to T$. However, the estimator $\hat{\alpha}_t$ is not asymptotically consistent when $\RE(\alpha) \in (\frac{1}{2}, H) $.
\end{theorem}
As an application, we will use $\hat{\alpha}_t$ to construct a consistent estimator $\tilde{\alpha}_{n}$ when high frequency data is available. We assume that the $\alpha$-fractional Brownian bridge given by equation \eqref{COMPLEX ALPHA BRIDGE} can be observed at discrete time points $\set{s_{n,k}:k=0,1,\dots, n}$ with $0=s_{n,0}<s_{n,1}<\cdots< s_{n,n}=t$. Set \begin{align}\label{lisan guji}
  \tilde{\alpha}_{n}=-\frac{\sum\limits_{k=1}^n \frac{\overline{Z}_{s_{n,k}}}{T-s_{n,k}}(Z_{s_{n,k}}-Z_{s_{n,k-1}})}{\sum\limits_{k=1}^n\frac{|Z_{s_{n,k}}|^2}{(T-{s_{n,k}})^2} ({s_{n,k}}-s_{n,k-1})}.  
\end{align}
We have the following asymptotic result for the estimator $\tilde{\alpha}_{n}$.
\begin{corollary}
     Let $H\in (\frac12, 1)$. Assume that $h_n:=\max_{1\le k\le n} ({s_{n,k}}-s_{n,k-1})\to 0$ and $s_{n,n}\to T$ as $n\to \infty$. When $\RE(\alpha) \in (0, \frac{1}{2}] $, we have $\tilde{\alpha}_n \xrightarrow{a.s.} \alpha$ as $n \to \infty$. However, the estimator $\tilde{\alpha}_n$ is not asymptotically consistent when $\RE(\alpha) \in (\frac{1}{2}, H) $.
\end{corollary}

The third contribution of this paper is to find the limiting distribution for the term $\hat{\alpha}_t - \alpha$ after it is appropriately normalized. The results are presented in Theorem~\ref{main result2}. We first introduce the complex Gaussian random variable $\eta$ whose real and imaginary parts are independent and normally distributed with mean zero and variance $\frac12 \sigma^2$. We denote by $\eta \sim CN(0,\sigma^2)$. The ratio of two independent complex Gaussian random variables $Z=\frac{\eta_1}{\eta_2}=Z_I+\mi Z_Q$ with $\eta_1\sim CN(0, \sigma_x^2)$ and $\eta_2\sim CN(0, \sigma_y^2)$, follows Cauchy distribution and has a density function \begin{align}\label{COMPLEX RV RATIO}
    f(z):=f(z_r, z_i)=\frac{\sigma_x^2}{\pi  \sigma_y^2}\left(|z|^2+\frac{\sigma_x^2}{\sigma_y^2} \right)^{-2},
\end{align}  where $z=z_r+\mi z_i\in \Cnum$. We write $Z\sim CR(\sigma_x^2/\sigma_y^2).$ 
However, both $Z_I$ and $ Z_Q$ have the same density function $$\frac12\frac{\sigma_x^2}{\sigma_y^2} \left(x^2+\frac{\sigma_x^2}{\sigma_y^2} \right)^{-\frac32},\quad x\in \Rnum,$$ so they do not follow Cauchy distribution any more.

\begin{theorem}\label{main result2}
  Assume $H\in (\frac12, 1)$. We have the following asymptotic results.
  \begin{itemize}
       \item [(i)] When  $\lambda  \in(0, 1-H)$, as $t \to T$, we have
       $$ (T-t)^{\lambda-H}(\alpha-\hat{\alpha}_t)\xrightarrow{d} (1-2\lambda) \times CR(\sigma_x^2/ \sigma_y^2),$$ 
       where \begin{align} \label{bizhi xy}
     \frac{\sigma_x^2}{\sigma_y^2}&=T^{2\lambda-2H}\frac{H-\lambda}{1-H-\lambda}\RE \left(\frac{\Gamma(2-\alpha-2H)}{\Gamma(1-\alpha)} \right)/\RE\left(\frac{\Gamma(1-\alpha)}{\Gamma(2H-\alpha)}\right).
 \end{align}
\item [(ii)] When $\lambda =1-H$, as $t \to T$, we have
    $$\frac{(T-t)^{1-2H}}{\sqrt{2|\log(T-t)|}}(\alpha -\hat{\alpha}_t) \xrightarrow{d}(2H-1)^2\times CR(\sigma_x^2/\sigma_y^2),$$
    where 
    $$\sigma_x^2/ \sigma_y^2=T^{2 (1-2H)} \RE \left(\frac{\Gamma(2-\alpha-2H)}{\Gamma(1-\alpha)} \right)/\RE\left(\frac{\Gamma(1-\alpha)}{\Gamma(2H-\alpha)}\right).$$
    \item [(iii)] When $\lambda \in (1-H,\frac12)$, as $t \to T$, we have $$\frac{(T-t)^{2\lambda-1}}{1-2\lambda}(\alpha-\hat{\alpha}_t)\xrightarrow{d} \frac{F_{\infty}+ {B(2H-1,\bar{\alpha})} H T^{2H-1}}{|\omega_T|^2},$$ where the $(1,1)$-th Wiener chaos random variable $F_{\infty}$ is defined by Proposition~\ref{zengliang kongzhi}.
    \item [(iv)] When $\lambda =\frac12$, as $t \to T$, we have $$|\log(T-t)|(\alpha-\hat{\alpha}_t)\xrightarrow{d} \frac{F_{\infty}+ {B(2H-1,\bar{\alpha})} H T^{2H-1}}{|\omega_T|^2},$$ where the $(1,1)$-th Wiener chaos random variable $F_{\infty}$ is defined by Proposition~\ref{zengliang kongzhi}.
   \end{itemize}
\end{theorem}

As mentioned, the statistical estimation for real-valued fractional Brownian bridge was first considered in \cite{EN}. However, it is not straightforward to directly extend the real-valued framework to the two-dimensional setting. In contrast to the main results in \cite{EN}, when $\lambda\in (0, 1-H]$, the two-dimensional limiting distribution whose density is given by \eqref{COMPLEX RV RATIO} is no longer a Cauchy distribution, as it is in the real-valued case, even when the parameter $w$ degenerates to zero.

The paper is organized as follows. In Section 2, we introduce some elements about the complex isonormal Gaussian process and the theory of complex Malliavin calculus and present some path regularities concerning the complex Wiener integral $\omega$ given as in \eqref{omega bds}. Some results regarding real-valued fractional Brownian bridge will be recalled. In Section 3, we will prove Theorem~\ref{main result} by means of Garsia-Rodemich-Rumsey inequality, different from the method of \cite{EN}.
In Section 4, we will prove Theorem~\ref{main result2}.
Two technical inequalities are provided in Appendix. Throughout the paper, we denote by $C$ a generic positive constant whose values may differ from line to line.

\section{Preliminary}\label{prelim}

The fBm $B^H = \set{B^H_t , t \in[0,T]}$ with Hurst
parameter $H \in (0, 1)$ is a continuous centered Gaussian process, defined on a complete probability
space $(\Omega,\mathcal{F}, P )$, with covariance function given by
$$\E(B^H_t B^H_s)=R_H(t,s)=\frac12 \big(\abs{t}^{2H}+\abs{s}^{2H}-\abs{t-s}^{2H}\big).$$
Let $\mathcal{E}$ denote the space of all real valued step functions on $[0,T]$. The Hilbert space $\mathfrak{H}$ is defined
as the closure of $\mathcal{E}$ endowed with the inner product
\begin{align*}
\innp{\mathbf{1}_{[0,s]},\,\mathbf{1}_{[0,t]}}_{\FH}=\E\big(B^H_s B^H_t \big).
\end{align*} 
Denote the isonormal process on the same probability space  by
$$B^H=\left\{B^H(h)=\int_{[0,T]}h(t)\dif B^H_t, \quad h \in \mathfrak{H}\right\}.$$ It is indexed by the elements in the Hilbert space $\mathfrak{H}$ and satisfies the It\^{o}'s isometry:
\begin{align}\label{G extension defn}
\mathbb{E}(B^H(g)B^H(h)) = \langle g, h \rangle_{\mathfrak{H}}, \quad
\forall g, h \in \mathfrak{H}. 
\end{align} 
If $H\in (\frac12,1)$ or the intersection of the supports of two elements $f,\,g \in \mathfrak{H}$ is of Lebesgue measure zero, we have
 \begin{align} \label{innp fg3-zhicheng0}
\langle f,\,g \rangle_{\FH}=H(2H-1)\int_{[0,T]^2}  f(t)g(s) \abs{t-s}^{2H-2} \dif t  \dif s. \end{align} 
Note that Theorem~\ref{main result} and Theorem~\ref{main result2} assume $H>1/2$. However, Theorem~\ref{fBr existence} assumes $H\in (0,1)$ for which the following proposition in \cite{AACZ} provides a method for computing the inner product of two functions in $\mathfrak{H}$ when $H \in (0, \frac{1}{2})$.
\begin{proposition}\label{st1-thm-original}
 Denote $\mathcal{V}_{[0,T]}$ as the set of bounded variation functions on $[0,T]$. Let $H \in (0, \frac{1}{2})$. For any two functions in the set $\mathcal{V}_{[0,T]}$, their inner product in the Hilbert space $\mathfrak{H}$ can be expressed as
 \begin{align} 
\langle f,\,g \rangle_{\FH}
&=H \int_{[0,T]^2}  f(t)  \abs{t-s}^{2H-1}\sgn(t-s) \dif t \nu_{g}(\dif s),\quad  \forall f,\, g\in \mathcal{V}_{[0,T]}, \label{neiji3}
\end{align} where $\nu_{g}(\dif s):=\dif \nu_{g}(s)$, and $\nu_g$ is the restriction on $\left({[0,T]},\mathcal{B}({[0,T]})\right)$ of the signed \textnormal{Lebesgue-Stieljes} measure $\mu_{g^0}$ on $\left(\Rnum,\mathcal{B}(\Rnum)\right)$ such that
\begin{equation*}
g^0(x)=\left\{
    \begin{array}{ll}
g(x), & \quad \text{if}~x\in [0,T],\\
0, &\quad \text{otherwise}.
 \end{array}
\right.
\end{equation*}
If $ g'(\cdot) $ is interpreted as the distributional derivative of $g(\cdot)$, then the formula \eqref{neiji3} admits the following representation:
 \begin{align} 
\langle f,\,g \rangle_{\FH}
&=H \int_{[0,T]^2}  f(t) g'(s) \abs{t-s}^{2H-1}\sgn(t-s) \dif t  \dif s,\quad  \forall f,\, g\in \mathcal{V}_{[0,T]}. \label{innp fg3-00}
\end{align}  
\end{proposition}

\subsection{Complex Wiener-It\^o multiple integrals}
In this subsection, we will introduce some essential definitions for understanding complex Wiener-It\^o integrals. Using the standard complexification method, we can complexify the Hilbert space $\FH$ and still denote it by $\FH$, a complex separable Hilbert space with inner product $\innp{\cdot,\cdot}_{\FH}$, linear in the first argument, and conjugate linear in the second argument.
\begin{definition}
  Let $z = x+ \mi y$ where $x, y\in\Rnum$. Complex Hermite polynomials $J_{m,n}(z)$ are uniquely determined by its generating function:
 \begin{align*}
     \exp\set{\lambda \bar{z}+\bar{\lambda} z -2\abs{\lambda}^2}=\sum_{m=0}^{\infty}\sum_{n=0}^{\infty}\frac{\bar{\lambda}^m\lambda^n}{m!n!}J_{m,n}(z).
 \end{align*}
\end{definition} 

\begin{definition}
    A complex Gaussian isonormal process $\set{Z(h):\,h\in \FH}$ over the complex Hilbert space $\FH$, is a centered and symmetric complex Gaussian family in $L^2(\Omega)$ such that 
\begin{align*}
    \E[Z(h)]=0,\quad \E[Z(g)\overline{Z(h)}]=\innp{g, h}_{\FH}, \quad \forall g, h\in \FH. 
\end{align*}
\end{definition}
\begin{definition}\label{def.complex.ito}
    For each $m, n \ge 0$, let $\mathcal{H}_{m,n}$ indicate the closed linear subspace of $L^2(\Omega)$ generated by the random variables of the type    \begin{align*}
    \set{J_{m,n}(Z(h)):\, h\in \FH, \, \norm{h}_{\FH}=\sqrt{2}}.
\end{align*}
The space $\mathcal{H}_{m,n}$ is called 
the $(m,n)$-th Wiener-It\^o chaos of $Z$.
\end{definition} 

\begin{definition}\label{chjfendy}
    For each $m, n \ge 0$, the linear mapping $$I_{m,n}(h^{\otimes m}\otimes \bar{h}^{\otimes n})=J_{m,n}(Z(h)),\quad h \in \mathfrak{H}$$ defines a complex Wiener-It\^o stochastic integral. The mapping $I_{m,n}$ provides a linear isometry between the space $\mathfrak{H}^{\odot m}\otimes \mathfrak{H}^{\odot n}$ (equipped with the norm $\frac{1}{\sqrt{m!n!}}\|\cdot\|_{\mathfrak{H}^{\otimes (m+n)}}$) and the space $\mathcal{H}_{m,n}$. Here $\mathcal{H}_{0,0} = \mathbb{R}$ and $I_{0,0}(x)=x$ by convention.
\end{definition}

From the above definition, if $f\in \mathfrak{H}^{\odot m}\otimes \mathfrak{H}^{\odot n}$ and $g\in \mathfrak{H}^{\odot n}\otimes \mathfrak{H}^{\odot m}$ satisfies a conjugate symmetry relation
\begin{align}\label{chongjifengongegx}
    g(t_1,\dots, t_n; s_1,\dots, s_m)=\overline{f(s_1,\dots, s_m; t_1,\dots, t_n )},
\end{align}
their complex Wiener-It\^o stochastic integrals must satisfy
\begin{align}\label{duiouguanxi}
    \overline{I_{m,n}(f)}=I_{n,m}(g).
\end{align}
Moreover, the complex  Wiener-It\^o integrals satisfy the isometry property. Namely, for $f\in\mathfrak{H}^{\odot a}\otimes\mathfrak{H}^{\odot b}$ and $g\in\mathfrak{H}^{\odot c}\otimes\mathfrak{H}^{\odot d}$,
\begin{equation}\label{isometry property}
	\E\left[ I_{a,b}(f)\overline{{I}_{c,d}(g)}\right]=\mathbf{1}_{\left\lbrace a=c, \, b=d \right\rbrace }  a!b!\left\langle f, g\right\rangle _{\FH^{\otimes (a+b)}}.
\end{equation}

\begin{proposition}
Complex multiple Wiener-It\^o integrals have all moments that
satisfy the following hypercontractivity inequality
           \begin{equation*}
            [\Enum |I_{p, q} (f)|^r]^{\frac{1}{r}} \le (r-1)^{\frac{p+q}{2}} [\Enum |I_{p, q} (f)|^2]^{\frac{1}{2}},  \qquad r \geq 2, 
            \label{hyper inequality}
            \end{equation*}where  $\abs{\cdot}$ is the absolute value (or modulus) of a complex number.
\end{proposition}

\begin{definition}
When $\FH=L^2(A,\mathcal{B}, \nu)$ with $\nu$ non-atomic, the $(i, j)$ contraction of two  symmetric functions $f\in \mathfrak{H}^{\odot a}\otimes \mathfrak{H}^{\odot b},\, g\in \mathfrak{H}^{\odot c}\otimes \mathfrak{H}^{\odot d} $ is defined as
\begin{align*}
  &\quad f\otimes_{i,j} g (t_1,\dots, t_{a+c-i-j};s_1,\dots, s_{b+d-i-j})\\
  &=\int_{A^{i+j}}\nu^{i+j}(\dif u_1\dots \dif u_i\dif v_1\dots \dif v_j)f(t_1,\dots, t_{a-i}, u_1, \dots , u_i; s_1,\dots, s_{b-j}, v_1,\dots, v_j)\\
  &\times g(t_{a-i+1},\dots, t_{a+c-i-j}, v_1,\dots, v_j; s_{b-j+1},\dots, s_{b+d-i-j}, u_1, \dots , u_i).
\end{align*} 
\end{definition}
The product formula for complex multiple Wiener-It\^o integrals states that for $f \in \mathfrak{H}^{\odot a} \otimes\mathfrak{H}^{\odot b}$ and $ g \in \mathfrak{H}^{\odot c} \otimes \mathfrak{H}^{\odot d}$, where $a, b, c, d \geq0$,
\begin{equation}\label{Product_formula}
	I_{a, b}(f) I_{c, d}(g)=\sum_{i=0}^{a \wedge d} \sum_{j=0}^{b \wedge c}\binom{a}{i} \binom{d}{i}\binom{b}{j}\binom{c}{j} i! j! I_{a+c-i-j, b+d-i-j}\left(f \otimes_{i, j} g\right).
\end{equation}

\subsection{Complex Malliavin calculus}
In this subsection, we will introduce some essential results for understanding the complex Malliavin calculus (please refer to \cite{CL}, \cite{CCL} and the references therein for more details).

Let $x_1,y_1,\ldots, x_m,y_m \in\mathbb{R}$ be $2m$ real variables. Define the complex variables $z_j = x_j + \mi y_j$ for $j=1,\ldots,m$. There exists a one-to-one correspondence between complex functions $g:\mathbb{C}^m\rightarrow \mathbb{C}$ and $\tilde{g}:\mathbb{R}^{2m}\rightarrow \mathbb{C}$ via the identity
  $$g\left(z_{1}, \ldots, z_{m}\right)=g\left(x_1+\mi y_1, \ldots, x_m+\mi y_m\right)=\tilde{g}\left(x_1, y_1, \ldots, x_m, y_m\right).$$
We say that $g\in C_p^{\infty}(\mathbb{C}^m)$, if the associated function $\tilde{g}\in C_p^{\infty}(\mathbb{R}^{2m}) $. That is,  $C_p^{\infty}(\mathbb{C}^m)$ is the set of all infinitely continuously differentiable complex functions of $2m$ real variables such that all Wirtinger partial derivatives
\begin{align}
		\partial_{j} g&=\frac{\partial}{\partial z_{j}} g\left(z_{1}, \ldots, z_{m}\right)=\frac{1}{ 2}\left( \frac{\partial}{\partial x_j}-\mi \frac{\partial}{\partial y_j}\right)\tilde{g}\left(x_{1}, y_1,\ldots, x_{m}, y_m\right),\label{Wirtinger derivatives1}\\
		\bar{\partial}_{j} g&=\frac{\partial}{\partial \bar{z}_{j}} g\left(z_{1}, \ldots, z_{m}\right)=\frac{1}{ 2}\left(\frac{\partial}{\partial x_j}+\mi \frac{\partial}{\partial y_j}\right)\tilde{g}\left(x_{1}, y_1,\ldots, x_{m}, y_m\right),\label{Wirtinger derivatives2}
\end{align}
have polynomial growth. 

The complex Malliavin derivative operators $\D$ and $\bar{\D}$ are defined as follows. Let $\mathcal{S}$ denote the set of all smooth random variables of the form
	\begin{equation}\label{complex smooth r.v.}
		G=g\left(Z\left({h}_{1}\right), \cdots, Z\left({h}_{m}\right)\right),
	\end{equation}
	where $ {h}_1,\dots, {h}_m\in\mathfrak{H}$, $m\geq1$ and $g\in C_p^{\infty}(\mathbb{C}^m)$.
	If $G \in \mathcal{S}$ has the form \eqref{complex smooth r.v.}, then the complex Malliavin derivatives of $G$ are the $\mathfrak{H}$-valued elements in $L^{2}(\Omega; \mathfrak{H} )$ defined by
	\begin{align}\label{definition of complex derivative}
		\D G &=\sum_{j=1}^{m} \partial_{j} g\left(Z\left( {h}_{1}\right), \cdots, Z\left( {h}_{m}\right)\right)  {h}_{j}, \\
		\bar{\D} G &=\sum_{j=1}^{m} \bar{\partial}_{j} g\left(Z\left( {h}_{1}\right), \cdots, Z\left( {h}_{m}\right)\right) \overline{ {h}}_{j}.
	\end{align}
The operators $\D$ and $\bar{\D}$ can be iterated such that $\D^{p} \bar{\D}^{q} G$ is a random variable taking values in $\mathfrak{H} ^{\odot p} \otimes \mathfrak{H} ^{\odot q}$ for any $G \in \mathcal{S}$. For $p+q\geq1$, the operators $\D^{p} \bar{\D}^{q}$ are closable from $L^{r}(\Omega)$ to $L^{r}\left(\Omega, \mathfrak{H}^{\odot p} \otimes \mathfrak{H}^{\odot q}\right)$ for every $r \geq 1$. Let $\mathscr{D}^{p, r} \cap\bar{\mathscr{D}}^{q, r}$ denote the closure of $\mathcal{S}$ with respect to the Sobolev seminorm $\|\cdot\|_{p, q, r}$ given by 
	\begin{equation*}	\|G\|_{p,q,r}^{r}=\sum_{i=0}^{p}\sum_{j=0}^{q}\mathbb{E}\left(\left\|\D^i\bar{\D}^jG\right\|^{r}_{\mathfrak{H}_{\mathbb{C}}^{\otimes (i+j)}}\right).
	\end{equation*}
 For $f\in\mathfrak{H}^{\odot p} \otimes \mathfrak{H}^{\odot q}$, we have $I_{p,q}(f)\in \bigcap_{p=0}^{\infty}\bigcap_{q=0}^{\infty}\bigcap_{r=1}^{\infty}\left( \mathscr{D}^{p, r} \cap\bar{\mathscr{D}}^{q, r}\right) $ and for any $k\geq0$,
 \begin{equation}
 	\D^k I_{p,q}(f)  =
 	\begin{cases}
 		\frac{p!}{(p-k)!} I_{p-k,q}(f),  & {\rm if} \ {k\leq p,}\\
 		0,  & {\rm if} \ k>p,
 	\end{cases}
 \end{equation} 
and 
\begin{equation}\label{bard bds}
	\bar{\D}^k I_{p,q}(f)  =
	\begin{cases}
		\frac{q!}{(q-k)!}I_{p,q-k}(f),  & {\rm if} \ {k\leq q,}\\
		0,  & {\rm if} \ k>q.
	\end{cases}
\end{equation} 

The divergence operators $\delta$ and $\bar{\delta} $ are defined as the adjoint of the Malliavin derivative operators $\D$ and $\bar{\D}$, respectively. The domain $\mathrm{Dom}( \delta) $ is the subset of $L^2(\Omega,\FH)$ consisting of those elements $u$ for which there exists a constant $c>0$ such that, for all $ F\in \mathcal{S}$,
   \begin{align*}
     \big|E[\innp{D F,u}_{\mathfrak{H}}]\big|\le c\norm{F}_{L^2}. 
   \end{align*}
The domain $\mathrm{Dom}(\bar{\delta})$ is defined analogously by replacing $D$ with $\bar D$.
If $u\in\mathrm{Dom}(\delta)$, then $ \delta (u)$ (resp. $\bar{\delta}(u)$) is an element in $L^2(\Omega)$ uniquely determined by the following duality formula,
   \begin{align}\label{dul form}
      E[\delta (u) \bar{F}]=E[\innp{u,\, D F}_{\mathfrak{H}}],
   \end{align}
for all $ F\in \mathcal{S}$ and an analogous identity holds for $\bar{\delta}$, i.e., $E[\bar{\delta} ( u) \bar{F}]=E[\innp{u,\,\bar{D} F}_{\mathfrak{H}}]$.

Similarly, for any $p,\,q\in \N$, we can define the divergence operator $\delta^p\bar{\delta}^q$. Then we have $$I_{p, q}(f)=\delta^p\bar{\delta}^q (f),\,\,\,\,\forall f \in\mathfrak{H}^{\odot p}\otimes\mathfrak{H}^{\odot q}.$$

\begin{proposition}\label{sandu-1} 
 If $F\in \mathscr{D}^{1, 2}$ and $u\in\mathrm{Dom}( \delta )$, or $F\in \bar{\mathscr{D}}^{1, 2}$ and $u\in\mathrm{Dom}( \bar{\delta} )$, then we have
   \begin{align*}
       \delta(\bar{F}u)&=\bar{F}\delta(u)-\innp{u,  DF}_{\FH}, \ {\rm and}\\
       \bar{\delta}({F}u)&={F}\bar{\delta}(u)-\innp{u,  \bar{D}\bar{F}}_{\FH}.
   \end{align*}
 \end{proposition}
Denote by $\zeta_t$ the complex fBm with the Hurst index $H\in (\frac12, 1)$ throughout the remainder of this subsection. Proposition~\ref{sandu-1} implies two propositions concerning stochastic integrals with respect to $\zeta_t$.
 \begin{proposition}
  For a function $F:\Cnum \to \Cnum$ in the set $ C_p^{2}(\mathbb{C})$, the following chain rule holds:
     \begin{equation*}
         F(\zeta_t)=F(0)+\int_0^t \partial F(\zeta_s)\delta \zeta_s +\int_0^t \bar{\partial} F(\zeta_s) \bar{\delta}\zeta_s+ 2H \int_0^t \partial \bar{\partial} F(\zeta_s) s^{2H-1}\dif s, 
     \end{equation*}where $\partial, \,  \bar{\partial}$ are the Wirtinger derivatives and the stochastic integrals with respect to the fBm $\zeta$ are the Skorohod divergence integrals.
 \end{proposition}

\noindent We use $\int_0^t u_s \dif \zeta_s$ and $\int_0^t u_s \delta \zeta_s$ to denote the Young integral and Skorohod integral of an element $u$ with respect to the fBm $\zeta_s$, respectively. Similar notations are applied for the integrals with respect to fBm $\bar\zeta_s$. The next proposition presents the relationship between the Young and Skorohod integrals.

\begin{proposition}\label{fuguanxi 2-3}
 Let $u=(u_t)_{t\in [0,T]}$ be a process with paths in $C^{\gamma}[0,T]$ where $\gamma>1-H$. If the process $u_t$ satisfies the conditions
 \begin{itemize}
     \item [(1)] $u_t\in \bar{\mathscr{D}} $ for all $t\in [0,T]$, and
     \item [(2)] $\int_0^T \int_0^T \abs{\bar{\D}_s u_x} |x-s|^{2H-2} \dif s \dif x<\infty$ a.s..
 \end{itemize}
 Then for all $t\in [0,T]$,
    \begin{align*}
       \int_0^t u_s \dif \zeta_s &=  \int_0^t u_s \delta \zeta_s +H(2H-1)\int_0^t \int_0^t \bar{\D}_s u_x |x-s|^{2H-2} \dif s \dif x.
    \end{align*}
If the process $u_t$ satisfies the conditions
 \begin{itemize}
     \item [(3)] $u_t\in \mathscr{D}$ for all $t\in [0,T]$, and
     \item [(4)] $\int_0^T \int_0^T \abs{{\D}_s u_x} |x-s|^{2H-2} \dif s \dif x<\infty$ a.s..
 \end{itemize}
Then for all $t\in [0,T]$,
    \begin{align*}
       \int_0^t u_s \dif \bar{\zeta}_s &=  \int_0^t u_s \bar{\delta} \zeta_s +H(2H-1)\int_0^t \int_0^t {\D}_s u_x |x-s|^{2H-2} \dif s \dif x.    
    \end{align*}
Moreover, the set of conditions (1)-(2) is equivalent to (3)-(4).
\end{proposition} 
The proofs of Propositions~\ref{sandu-1}-\ref{fuguanxi 2-3} are similar to the results of the real Malliavin calculus.  Please refer to \cite{Nua} for more details.

\subsection{Results regarding a real-valued fractional Brownian bridge}

In this section, let us recall some known results in the literature concerning the real-valued $\alpha$-fractional Brownian bridge given by \eqref{real ALPHA BRIDGE}. These results will be used in our proofs of the main theorems. 

Denote the Wiener integral
\begin{align}\label{xit dingyi}
\xi_t=\int_0^t (T-u)^{-\alpha} \dif B^H_u,\qquad \quad 0\le t< T. 
\end{align} For any fixed $t<T$ and any $\epsilon>0$, the Gaussian process $\xi=(\xi_u)_{u\in[0,t]}$  admits a modification on $[0,t]$ with $(H-\epsilon)$-H\"{o}lder continuous paths.  The stochastic process $X_t$ defined by \eqref{xt bds} can be written as $X_t = (T-t)^\alpha \xi_t$. We have the following results about the processes $X_t$ and $\xi_t$ (see \cite{AACZ}).

\begin{itemize}
    \item When $\alpha\in (0,H)$, it is known that
$\xi_T:=\lim_{t\uparrow T}\xi_t$ exists in $L^2$ and almost surely, and 
\begin{align}\label{xiT2qiw}
\E[\xi_T^2]=\frac{\Gamma(1+2H)}{2(H-\alpha)}\frac{\Gamma(1-\alpha)}{\Gamma(2H-\alpha)} T^{2(H-\alpha)}.
\end{align} Moreover,
 if $\alpha\in (0,H)$, there exists a positive constant $C$ independent of $T$ such that
\begin{align}\label{gima2fang jie-1}
\sigma^2(s,t)=\E\big[(\xi_s-\xi_t)^2\big]\le C\abs{s-t}^{2(H-\alpha)},\quad 0\le s, t \le T.
\end{align}
By this fact, the Gaussian process $\xi_t$ can be extended to the interval $t \in [0, T]$ and admits a modification on $[0,T]$ with $(H-\alpha-\epsilon)$-H\"{o}lder continuous paths for any $0< \epsilon < H-\alpha$.  Consequently, we define $X_T := \lim_{t \to T} X_t = 0$ and the process $X:=(X_t)_{t\in[0,T]}$ is called a $\alpha$-fractional Brownian bridge under the condition of $\alpha \in (0, H)$.
    \item When $\alpha\in (H,1)$, we denote
\begin{align}\label{eta dingyi}
\tilde{X}_t:=\frac{X_t}{(T-t)^H}=(T-t)^{\alpha-H} \int_0^t (T-u)^{-\alpha} \dif B^H_u,\qquad \quad 0\le t< T. 
\end{align}
We find that $\tilde{X}_T:=\lim_{t\uparrow T}\tilde{X}_t$ exists in $L^2$ and 
\begin{align}
\E[\tilde{X}_T^2] =\frac{\alpha H}{\alpha-H}B(2H,\,1+\alpha-2H),\label{zetaT2qiw}
\end{align} 
where $B(\cdot, \cdot)$ is the Beta function. 
\item When $\alpha=H$, the process $\tilde{X}_t:=\frac{\xi_t}{- \log(T-t)}, 0\le t< T,$ satisfies the fact that $\tilde{X}_T:=\lim_{t\uparrow T}\tilde{X}_t$ exists in $L^2$ and
\begin{align}
\E[\tilde{X}_T^2] 
=2H^2 B(2H,\,1-H).\label{zetaT2qiwlinjie}
\end{align}
\end{itemize}

\subsection{Existence of complex $\alpha$-fractional Brownian bridge}\label{sec.2.3}

Based on previous results, we will prove the well-posedness of the complex $\alpha$-fractional Brownian bridge. Assume that the Hurst parameter $ H\in(0, 1)$ and $\omega=(\omega_t)_{t\in[0,T)}$ is the complex Wiener integral given as in \eqref{omega bds}. We first present the results about path regularities of the process $\omega$.

\begin{lemma}\label{xi decompose}
  Assume $ \RE(\alpha)\in(0, 1)$. Then  we have that for all $0\le s\le t < T$,
    \begin{align}\label{xi decompose4}
	\frac{1}{H}\Enum [|\omega_s-\omega_t|^2]=J_1(s, t)+J_2(s, t)+J_3(s, t)-\frac{\mi \bar{\alpha}}{H}\IM({J_1(s, t)}),
    \end{align}
  where
    \begin{align}
	&J_1(s, t)=2H\int_s^t (T-v)^{-\alpha-1}d v\int_v^t (T-u)^{-\bar{\alpha}}(u-v)^{2H-1}\dif u, \label{zhongjianzhy}\\
	&J_2(s, t)=(T-t)^{1-\bar{\alpha}}\int_s^t (T-u)^{-\alpha-1}(t-u)^{2H-1} \dif  u, \\
	&J_3(s, t)=(T-s)^{-\bar{\alpha}}\int_s^t (T-u)^{-\alpha}(u-s)^{2H-1} \dif  u.\label{zhongjzhy 2}
    \end{align}
\end{lemma} 
The proof is similar to that of Lemma 5.1 in \cite{AACZ}.

\begin{proposition}\label{dianzeduliang}
Assume $ \RE(\alpha)\in (0,1)$. For any fixed $t\in (0,T)$, there exists a positive constant $C_{t,T}$ depending on $t,\,T$ such that
\begin{align}\label{gima2fang jie}
\sigma^2(u,v):=\E\big[\abs{\omega_u-\omega_v}^2\big]\le C_{t,T}\abs{u-v}^{2H },\quad 0\le u, v\le t.
\end{align} Hence the complex Gaussian process $\omega$ admits a modification on $[0,t]$ with $(H-\epsilon)$-H\"{o}lder continuous paths for any fixed $t<T$ and any $0<\epsilon<H$.

Moreover, if $ \lambda=\RE(\alpha) \in (0, H)$, then there exists a positive constant $C$ independent of $T$ such that
  \begin{align}\label{gima2fang jie}
      \sigma^2(s,t)=\E\big[\abs{\omega_s-\omega_t}^2\big]\le C\abs{s-t}^{2(H-\lambda)},\quad 0\le s, t< T.
  \end{align}
\end{proposition}

Next, we study the limiting behavior of the process $\omega_t$ with appropriate normalization.
\begin{proposition}\label{prop Y}
Assume $\RE(\alpha) \in [H, 1)$. Define a scaling limit of $\omega_t$ as
			\begin{align}\label{Y_t}
				Y_t=\begin{cases}
                \frac{\omega_t}{\sqrt{2|\log(T-t)|}} , \qquad & {\rm if} \ \RE(\alpha)=H;\\
				\sqrt{\RE(\alpha) -H}\,(T-t)^{\RE(\alpha)-H}  \omega_t,\qquad& {\rm if} \ \RE(\alpha) \in(H,1).
				\end{cases}
			\end{align}
We have that $Y_T: =\lim_{t\to T}Y_t$ exists in $L^2$ and that 
	\begin{align}
		\Enum [|Y_T|^2]
        = \frac12{\Gamma(1+2H)} \RE\left(\frac{\Gamma(1+\alpha-2H)}{\Gamma(\alpha)}\right)
	\end{align}
and
    $$\Enum [\zeta_s\overline{Y_T}]=0, \qquad \forall s\in[0, T].$$
\end{proposition}
Proofs of Propositions~\ref{dianzeduliang}-\ref{prop Y} are similar to those of Proposition 5.2 and Theorems 1.4-1.5 in \cite{AACZ}. To conclude this section, we will show the well-posedness of the complex $\alpha$-fractional Brownian bridge as stated in Theorem~\ref{fBr existence}.

{\bf Proof of Theorem~\ref{fBr existence} : }
It follows from equation \eqref{xi decompose4} that 
\begin{align}
\E[\xi_T^2]=\lim_{t\uparrow T}\E[\xi_t^2] =H\times \left[J_1(0,T)+J_2(0,T)+J_3(0,T)-\frac{\mi \bar{\alpha}}{H}\IM({J_1(0, T)}) \right],\label{qidiands}
\end{align}where $J_1,J_2, J_3$ are given in equations \eqref{zhongjianzhy}-\eqref{zhongjzhy 2}.
It is evident that
\begin{align}
J_3(0,T)&=T^{-\bar{\alpha}}  \int_{0}^T(T-u)^{-\alpha}  u^{2H-1}  \dif u=B(1-\alpha, 2H)T^{2(H-\lambda)},\label{i3tjixian}
\end{align}
and
\begin{align}
\frac{1}{2H}J_1(0,T)&=\int_{0}^{ T}(T-u)^{-\alpha-1}\dif u \int_{u}^T (T-v)^{-\bar{\alpha}}  (v-u)^{2H-1} \dif v\notag \\
&=\frac{B(1-\bar{\alpha}, 2H)}{2(H-\lambda)}T^{2(H-\lambda)}.\label{J1tjixian}
\end{align}
By the change of variables $y=T-t,\,x=t-u$, and Lebesgue's dominated theorem,
\begin{align}
\abs{J_2(0,T)}&=\lim_{t\uparrow T} (T-t)^{1- {\lambda}}\int_{0}^{ t}(T-u)^{-\lambda-1} (t-u)^{2H-1}\dif u\notag \\
&=\lim_{y\to 0+} y^{1-\lambda}\int_{0}^{ T}(y+x)^{-\lambda-1} x^{2H-1}\mathbf{1}_{[0,T-y]}(x)\dif x\notag \\
&=\lim_{y\to 0+}\int_{0}^{ T}\left(\frac{y}{y+x}\right)^{1-\lambda} \left(\frac{x}{y+x}\right)^{2\lambda}  x^{2(H-\lambda)-1} \dif x=0\label{j0tjixian}
\end{align} when $\lambda  \in (0, H)$.
 Plugging the facts \eqref{i3tjixian}-\eqref{j0tjixian} into equation \eqref{qidiands}, we obtain the desired result \eqref{wt2 norm}. Finally, by applying Proposition~\ref{dianzeduliang} and the Kolmogorov-Centsov theorem, we have that when $\lambda \in (0, H)$, the Gaussian process $\omega:=\{\omega_t, t \in [0, T]\}$ admits a modification on $[0,T] $ with $(H-\lambda-\epsilon)$-H\"{o}lder continuous paths for all $\epsilon\in (0,H-\lambda)$.
 {\hfill\large{$\Box$}} 

\section{The strong consistency of the LSE}

In this section, we prove Theorem~\ref{main result} to show that the LSE of parameter $\alpha$ is strongly consistent. To achieve this, we rewrite equation \eqref{ALPHA LSE} as a ratio process in terms of complex Young integrals:
\begin{align}\label{qidian}
    \hat{\alpha}_t -\alpha =\frac{\int_{0}^t (T-u)^{\bar{\alpha}-1}\bar{\omega}_u\dif \zeta_u }{\int_0^t (T-u)^{2\lambda-2}|\omega_u|^2 \dif  u}.
\end{align}
We must study the limiting behavior of both the numerator and the denominator of the ratio process.  

We first get some results regarding two specific complex Wiener-It\^o chaos processes defined as follows. Denote
    \begin{align}\label{hehanshu01}
	\psi_t(u, v)=(T-u)^{\overline{\alpha}-1}(T-v)^{-\overline{\alpha}}\mathbf{1}_{\{0\leq v\leq u\leq t\}}.
    \end{align}
Define the complex Wiener-It\^o chaos process 
$G=(G_t)_{t\in[0,T)}:=\left(I_{1,1}(\psi_{t} )\right)_{t\in[0,T)}.$ 
Making the change of variable $s=\frac{1}{T-t}$ yields another complex Wiener-It\^o chaos process 
\begin{align}\label{F bds}
    F=(F_s)_{s\in[\frac1T,\infty)}:= \left(I_{1,1}\left(\psi_{T-\frac{1}{s}}\right)\right)_{s\in[\frac1T,\infty)}.
\end{align} The limiting behavior and the regularity property of the two processes are presented in the following propositions.

\begin{proposition}\label{prop.3.2}
The complex Winere-It\^o chaos process $G:=( I_{1,1}(\psi_t ))_{t\in [0,T)}$, where the complex function $\psi_t$ is given by \eqref{hehanshu01}, satisfies
\begin{align*}
    \begin{cases}
     \limsup_{t\to T}(T-t)^{2(1-H-\lambda)}\E[\abs{G_t}^2] <\infty, \qquad & {\rm if} \ \RE(\alpha)\in (0, 1-H);\\
    \limsup_{t\to T}\frac{1}{-\log (T-t)} \E[\abs{G_t}^2] <\infty, \qquad & {\rm if} \ \RE(\alpha)=1-H;\\
	\lim_{t\to T}\E[\abs{G_t}^2]<\infty,\qquad& {\rm if} \ \RE(\alpha) \in(1-H,1).
	\end{cases}
\end{align*}
\end{proposition}
\begin{proof}
 Denote the constant $C_H=H(2H-1)$ where $H \in (\frac{1}{2}, 1)$. By It\^o's isometry \eqref{isometry property}, we have that 
\begin{align}
  \E[\abs{G_t}^2]&=C_H^2\int_{0\leq v_1\leq u_1\leq t,\,0\leq v_2\leq u_2\leq t}(T-u_1)^{\overline{\alpha}-1}(T-v_1)^{-\overline{\alpha}} (T-u_2)^{ {\alpha}-1}(T-v_2)^{- {\alpha}} \notag\\
  &\times \abs{u_1-u_2}^{2H-2}\abs{v_1-v_2}^{2H-2} \dif u_1\dif u_2 \dif v_1\dif v_2\label{bdshi gt2}.
\end{align} 
Case 1: Assume $\lambda\in (0,1-H)$. It follows from equation \eqref{bdshi gt2} that 
\begin{align}
  & \limsup_{t\to T} (T-t)^{2(1-H-\lambda)}\E[\abs{G_t}^2]\notag \\
  &\le C_H \lim_{t\to T} (T-t)^{2(1-H-\lambda)} \int_{(0,t)^2} (T-u_1)^{\lambda-1}(T-u_2)^{\lambda-1}\abs{u_1-u_2}^{2H-2}\dif u_1\dif u_2\notag\\
   &\times  C_H  \lim_{t\to T}   \int_{(0,t)^2} (T-v_1)^{-\lambda}(T-v_2)^{-\lambda}\abs{v_1-v_2}^{2H-2}\dif v_1\dif v_2\notag\\
   &=\frac{(1-\lambda)H}{1-\lambda-H}B(2H,\,2-\lambda-2H) \frac{\Gamma(2H)}{2(H-\lambda)}\frac{\Gamma(1-\lambda)}{\Gamma(2H-\lambda)} T^{2(H-\lambda)}, \label{gt2norm shangjie}
\end{align} where in the last line we use the results of \eqref{xiT2qiw} and \eqref{zetaT2qiw}, and It\^{o}'s isometry.\\
Case 2: Assume $\lambda=1-H$. Using results of \eqref{xiT2qiw} and \eqref{zetaT2qiwlinjie}, and It\^{o}'s isometry, we have 
\begin{align}
   \limsup_{t\to T}\frac{1}{-\log (T-t)} \E[\abs{G_t}^2] \le H^2 B(2H,\,1-H)\frac{\Gamma(2H)}{2H-1}\frac{\Gamma(H)}{\Gamma(3H-1)} T^{2(2H-1)}.
   \end{align}
Case 3: Let $\lambda\in(1-H, 1)$. Proposition~\ref{last estimeate} implies that 
$\lim_{t\to T}\E[\abs{G_t}^2]<\infty$.
\end{proof}

\begin{proposition}\label{zengliang kongzhi}
  For the complex Wiener-It\^o chaos process $F=(F_s)_{s\in[\frac1T,\infty)}$ defined by equation \eqref{F bds}, there exists a positive constant $C_T$  such that for all $\frac1T \le s\le t<\infty$,
  \begin{align}\label{cha zengliang}
      \E\left(\abs{F_t-F_s}^2\right) \le C_T \abs{t-s}^{1-\RE(\alpha)}.  \end{align} 
Moreover, when $\RE(\alpha) \in(1-H,H)$, it holds true that
\begin{align}\label{Finfty dyi}
    \lim\limits_{s\to \infty}  {F_s} =\lim\limits_{t\to T}  {G_t} < \infty
\end{align} in the sense of both $L^2(\Omega)$ and a.s.. Then we define the $(1,1)$-th Wiener chaos random variable $F_\infty$ by $F_\infty := \lim\limits_{s\to \infty}  {F_s}$.
\end{proposition}
\begin{proof} Denote $C_H=H(2H-1),\, \beta=2H-2,\,\dif \vec{u}=\dif u_1\dif u_2,\,\dif \vec{v}=\dif v_1\dif v_2$ and integration region
\begin{align*}
   \Delta'&=\set{T-\frac{1}{s}\le v_1\le u_1\le T-\frac{1}{t},T-\frac{1}{s}\le v_2\le u_2\le u_1\le  T-\frac{1}{t}}.
\end{align*}
It follows from the It\^o's isometry and the symmetry that
\begin{align}
    & \E[\abs{F_t-F_s}^2] \notag\\
    &=2C_H^2\RE\int_{\Delta'} (T-u_1)^{\bar{\alpha} -1}(T-v_1)^{-\bar{\alpha}} (T-u_2)^{{\alpha} -1}(T-v_2)^{-{\alpha}} (u_1-u_2)^{\beta}\abs{v_1-v_2}^{\beta} \dif \vec{u} \dif\vec{v}  \label{guji 54}\\
    &=2C_H^2\RE\int_{\Delta''} y_1^{1-\bar{\alpha}-2H} y_2^{1-{\alpha}-2H} (y_1-y_2)^{\beta} x_1^{\bar{\alpha}-2H}x_2^{{\alpha}-2H}\abs{x_1-x_2}^{\beta} \dif \vec{y} \dif\vec{x},\notag
\end{align}where in the last line, we have made the change of variables: $x_i=\frac{1}{T-v_i},\,y_i=\frac{1}{T-u_i},\,i=1,2$ and the notations
$\dif \vec{x}=\dif x_1\dif x_2,\,\dif \vec{y}=\dif y_1\dif y_2$
and
$$ \Delta''=\set{ {s}\le x_1\le y_1\le  {t}, {s}\le x_2\le y_2\le y_1\le   {t}}. $$
By the absolute value inequality and the symmetry, we have that for $s\ge \frac{1}{T}$,
\begin{align}
   \E\left(\abs{F_t-F_s}^2\right)
  &\le 4C_H^2\int_{s\le y_2\le y_1\le t} y_1^{1-{\lambda}-2H} y_2^{1-{\lambda}-2H} (y_1-y_2)^{\beta} \dif \vec{y}\notag\\
  & \quad \times \int_{\frac{1}{T}\le x_1\le x_2} x_1^{ {\lambda}-2H}x_2^{{\lambda}-2H} (x_2-x_1)^{\beta}  \dif\vec{x}\notag\\
  &\le C_T \int_{s\le y_2\le y_1\le t} y_1^{1-{\lambda}-2H} y_2^{1-{\lambda}-2H} (y_1-y_2)^{\beta} \dif \vec{y},\label{eq.58 zhongjian}
\end{align}where in the last line, we use the following fact:
when $0<\lambda<H,\,H\in (\frac12,1)$,
\begin{align*}
   \int_{\frac{1}{T}\le x_1\le x_2} x_1^{ {\lambda}-2H}x_2^{{\lambda}-2H} (x_2-x_1)^{\beta}  \dif\vec{x}<\frac{ B(1-H, 2H-1)}{H-\lambda}T^{2(H-\lambda)} .
\end{align*}
Since for $s\ge \frac{1}{T}$, we have that 
\begin{align}
    \int_{s\le y_2\le y_1\le t} y_1^{1-{\lambda}-2H} y_2^{1-{\lambda}-2H} (y_1-y_2)^{\beta} \dif \vec{y}&\le \int_{s}^t  y_1^{1-{\lambda}-2H}\dif y_1\int_{\frac{1}{T}}^{y_1} y_2^{1-{\lambda}-2H} (y_1-y_2)^{\beta} \dif {y}_2\notag\\
    &\le T^{\lambda+2H-1} \int_{s}^t  y_1^{1-{\lambda}-2H}\dif y_1\int_{\frac{1}{T}}^{y_1}  (y_1-y_2)^{\beta} \dif {y}_2\notag\\
    &\le \frac{T^{\lambda+2H-1}}{2H-1} \int_{s}^t  y_1^{ -{\lambda} }\dif y_1\notag\\
    &=\frac{T^{\lambda+2H-1}}{(2H-1)(1-\lambda)}[t^{1-\lambda}-s^{1-\lambda}]\notag\\
    &\le \frac{T^{\lambda+2H-1}}{(2H-1)(1-\lambda)} (t-s)^{1-\lambda},\label{jbbds-1}
\end{align}where in the last line, we use the inequality
$1-x^{\theta}\le (1-x)^{\theta},\quad \forall x\in (0,1), \,\theta\in (0,1)$.
Substituting the estimate \eqref{jbbds-1} into the estimate \eqref{eq.58 zhongjian}, we obtain the desired result \eqref{cha zengliang}.

Finally, Proposition~\ref{prop.3.2} implies that $\lim_{t\to T}\E[\abs{G_t}^2]<\infty$  when $\RE(\alpha) 
\in(1-H,H)$. Together with the Cauchy criterion,  the following estimate
\begin{align}
    \E[\abs{G_t-G_s}^2]\le C\times\abs{t-s}^{2(2H-1)} 
\end{align} implies the desired result \eqref{Finfty dyi} in the sense of both $L^2(\Omega)$ and a.s.. To show that the above estimate holds true, we denote 
\begin{align*}
    \Delta=\set{s\le v_1\le u_1\le t,s\le v_2\le u_2\le u_1\le  t},
\end{align*}
and compute, in a similar way to equation \eqref{guji 54} as follows.
\begin{align*}
  & \E[\abs{G_t-G_s}^2]\notag\\
  &=  2C_H^2\RE\int_{\Delta} (T-u_1)^{\bar{\alpha} -1}(T-v_1)^{-\bar{\alpha}} (T-u_2)^{{\alpha} -1}(T-v_2)^{-{\alpha}} (u_1-u_2)^{\beta}\abs{v_1-v_2}^{\beta} \dif \vec{u} \dif\vec{v}\notag\\
  &\le   2C_H^2 \int_{[s,t]^4} (T-u_1)^{\lambda -1}(T-v_1)^{-\lambda} (T-u_2)^{{\lambda} -1}(T-v_2)^{-{\lambda}} \abs{u_1-u_2}^{\beta}\abs{v_1-v_2}^{\beta} \dif \vec{u} \dif\vec{v}\notag\\
  &=2 \E[\xi_t-\xi_s]^2\E[\tilde{\xi}_t-\tilde{\xi}_s]^2\notag\\
  &\le C \abs{t-s}^{2(H-\lambda)}\times \abs{t-s}^{2\big(H-(1-\lambda)\big)}=C\abs{t-s}^{2(2H-1)}, 
\end{align*}  where $\xi_t$ and $\tilde{\xi}_t$ are the Gaussian processes given by equation \eqref{xit dingyi} with the parameter $\alpha$ taking $\lambda$ and $1-\lambda$ respectively, and in the last line we use the estimate \eqref{gima2fang jie-1}.
\end{proof}

To conclude this section, we will prove Theorem~\ref{main result}. Denote the complex Young integral in equation \eqref{qidian} by
\begin{align}
    \eta_t:=\int_{0}^t (T-u)^{\bar{\alpha}-1}\bar{\omega}_u\dif \zeta_u.
\end{align} 
Then the ratio process in equation \eqref{qidian} can be written as
\begin{align}\label{error.term}
    \hat{\alpha}_t -\alpha =\frac{C_{T,t}^{-1} \eta_t}{C_{T, t}^{-1} \int_0^t (T-u)^{2\lambda-2}|\omega_u|^2 \dif  u},
\end{align}
where the normalizing factor $C_{T,t}$ is defined by
  \begin{equation}
     C_{T,t} = \begin{cases}
      (T-t)^{2\lambda-1}, \ & {\rm if} \ \lambda \in\left(0,\frac12\right), \\
      \abs{\log(T-t)}, \ & {\rm if} \ \lambda = \frac12.
  \end{cases}
  \end{equation}
 
{\bf Proof of Theorem~\ref{main result}: }
 By the regularity of the process $\omega$, 
 we apply L'H\^{o}pital's rule as $t \to T$ to have the following results:
\begin{align}
 \lim_{t\to T} \frac{1}{(T-t)^{2\lambda-1}}   \int_0^t (T-u)^{2\lambda-2}|\omega_u|^2 \dif  u = \left(1-2\lambda\right)^{-1}\abs{\omega_T}^2, \, a.s. \qquad & {\rm if} \ \lambda \in\left(0,\frac12\right); \label{fenmu case0}\\
 \lim_{t\to T} \frac{1}{\abs{\log(T-t)}}   \int_0^t (T-u)^{2\lambda-2}|\omega_u|^2 \dif  u  = \abs{\omega_T}^2, \, a.s. \qquad & {\rm if} \ \lambda = \frac12. \label{fenmu case0-0}
\end{align}

To show $\hat\alpha_t - \alpha \to 0$ as $t \to T$, it suffices to prove $C^{-1}_{T, t}\eta_t \to 0$. First, we can verify that the conditions of Proposition~\ref{fuguanxi 2-3} are true for the Young integral $\eta_t$. In fact, by equation \eqref{bard bds}, we know that $$\bar{\D}_v \Big((T-u)^{\bar{\alpha}-1}\bar{\omega}_u\Big)= (T-u)^{\bar{\alpha}-1}(T-v)^{ \bar{\alpha}}\mathbf{1}_{\set{0<v<u}}.$$
Since $\lambda\in (0,1)$ and $(1-\lambda)+\lambda =1< 2H$, we have that for all $t\in(0,T)$,
    \begin{align*}
        & \int_0^t \int_0^t \abs{\bar{\D}_v \Big((T-u)^{\bar{\alpha}-1}\bar{\omega}_u\Big)} |u-v|^{2H-2} \dif u \dif v\\
        & \le  \int_0^T (T-u)^{\lambda-1} \dif u\int_0^u (T-v)^{-\lambda} (u-v)^{2H-2} \dif v  = C T^{2H-1},
    \end{align*}
where the last step is by Lemma~\ref{cal.result1}. Together with the regularity of process $\omega$, we find that all the conditions in Proposition~\ref{fuguanxi 2-3} are valid. Thus, by Proposition~\ref{fuguanxi 2-3}, we have that
    \begin{align}
        \eta_t & = \int_0^t\overline{\omega}_u (T-u)^{\overline{\alpha}-1} \delta \zeta_u+H(2H-1)\int_0^t (T-u)^{\overline{\alpha}-1} \int_0^u (T-v)^{-\overline{\alpha}} (u-v)^{2H-2}\dif u \dif v \notag\\
         & =I_{1,1}(\psi_t(u, v))+H(2H-1)\int_0^t  (T-u)^{\overline{\alpha}-1}\dif u \int_0^u (T-v)^{-\overline{\alpha}} (u-v)^{2H-2}\dif v.\label{etabiaosh}
    \end{align}

Second, by Lemma~\ref{cal.result1}, the second term of \eqref{etabiaosh} satisfies 
\begin{align}\label{eq. 42-bizhi}
     \lim_{t\to T}  C_{T,t}^{-1}  \int_0^t  (T-u)^{\overline{\alpha}-1}\dif u \int_0^u (T-v)^{-\overline{\alpha}} (u-v)^{2H-2}\dif v = 0.
\end{align} For the first term in equation \eqref{etabiaosh}, we claim that as $t\to T$,
\begin{equation}\label{i11 guji}
    C_{T, t}^{-1} I_{1,1}(\psi_t) = C_{T, t}^{-1} G_t \to 0.
\end{equation}

\noindent To show the convergence \eqref{i11 guji}, it is sufficient to prove that
\begin{align}\label{i11 guji0}
    \begin{cases}
    \lim_{s\to \infty} \frac{1 }{s^{1-2\lambda}} F_s =0, \, a.s. & {\rm if} \ \lambda \in(0,\frac12),\\ 
 \lim_{s\to \infty} \frac{1}{\log s} F_s =0, \, a.s. \quad & {\rm if} \ \lambda =\frac12.
 \end{cases}
\end{align} It is clear that 
\begin{align}
    \E[\abs{F_s}^2]=E[\abs{G_t}^2],\quad \text{ where} \quad s=\frac{1}{T-t}.
\end{align}
Hence, Proposition~\ref{prop.3.2}, Borel-Cantalli's lemma and the hypercontractivity inequality of the complex Wiener-It\^{o} integral imply that 
\begin{align}\label{i11 guji-zilie}
    \lim_{n\to \infty} \frac{1}{n^{1-2\lambda}} F_n &=0, \, a.s. \quad {\rm for} \ \lambda \in\left(0,\frac12\right).
\end{align}

It follows from Proposition~\ref{zengliang kongzhi}, the hypercontractivity of complex multiple Wiener-It\^o integrals and the Garsia-Rodemich-Rumsey inequality that for any real number $p > \frac{4}{1-\lambda}$, $q >1$ and any integer $n \ge 1\vee \frac{1}{T}$, 
\begin{align*}
\abs{F_t-F_s}\le R_{p,q} n^{q/p} , \qquad \forall \ t,s\in [n,n+1] ,
\end{align*}where $R_{p,q}$ is a random constant independent of $n$ (see \cite{CHW}). Since
\begin{align*}
\abs{\frac{F_s}{s^{1-2\lambda}}}\le \frac{1}{s^{1-2\lambda}}\abs{F_s-F_n}+ \left(\frac{n}{s}\right)^{1-2\lambda}\frac{\abs{F_n}}{n^{1-2\lambda}},
\end{align*} where $n=[s]$ is the biggest integer less than or equal to a real number $s$, we have $\frac{F_s}{s^{1-2\lambda}}$ converges to $0$ almost surely as $s\to \infty$ when $\lambda\in (0,\frac12)$. 

Finally, when $\lambda=\frac12$, Proposition~\ref{prop.3.2} and Proposition~\ref{zengliang kongzhi} imply that in the complex $(1,1)$-th Wiener chaos space of $\zeta$, we have $ \lim\limits_{s\to \infty} \E[\abs{F_s}^2]<\infty.$ This shows that the random variable $F_{\infty}:= \lim\limits_{s\to \infty} F_s$ is finite almost surely.  Hence, we have that $\lim\limits_{s\to \infty} \frac{F_s}{\log s }=0$ almost surely. This finishes the verification of $C_{T, t}^{-1} \eta_t \to 0$ as $t \to T$. Recalling equation \eqref{error.term}, we can conclude that the LSE $\hat\alpha_t$ is strongly consistent.

 {\hfill\large{$\Box$}} 
\begin{remark}\label{rem.3.0001}
From the above proof, when $\RE(\alpha) \in\left(\frac12,H\right)$, we have that      \begin{align*}
    \hat{\alpha}_t-\alpha  \xrightarrow{a.s.} \frac{F_{\infty}+ {B(2H-1,\bar{\alpha})} H T^{2H-1}}{\int_0^T (T-u)^{2\lambda-2}|\omega_u|^2 \dif  u}.
\end{align*} This implies that the LSE $\hat{\alpha}_t$ is no longer consistent when $\RE(\alpha) \in\left(\frac12,H\right)$.
\end{remark}

\section{The asymptotic distribution}

In this section, we will prove Theorem~\ref{main result2} on the asymptotic distribution of the term $\hat{\alpha}_t -\alpha$ after normalizing it in an appropriate way.

 {\bf Proof of Theorem~\ref{main result2}: } 
When $\lambda \in \left(1-H,\frac12\right)$, it follows from equations \eqref{qidian} and \eqref{etabiaosh}
that 
\begin{align*}
   (T-t)^{2\lambda-1}\big( \hat{\alpha}_t -\alpha) &=\frac{\int_{0}^t (T-u)^{\bar{\alpha}-1}\bar{\omega}_u\dif \zeta_u }{\frac{1}{(T-t)^{2\lambda-1}}\int_0^t (T-u)^{2\lambda-2}|\omega_u|^2 \dif  u} \\
   &=\frac{G_t+C_H \int_0^t  (T-u)^{\overline{\alpha}-1}\dif u \int_0^u (T-v)^{-\overline{\alpha}} (u-v)^{2H-2}\dif v }{\frac{1}{(T-t)^{2\lambda-1}}\int_0^t (T-u)^{2\lambda-2}|\omega_u|^2 \dif  u}\\
   &\xrightarrow{d} \frac{(1-2\lambda)[F_{\infty}+ {B(2H-1,\bar{\alpha})} H T^{2H-1} ]}{|\omega_T|^2},
\end{align*} where in the last line, we use the results \eqref{fenmu case0}, \eqref{Finfty dyi} and Lemma~\ref{cal.result1}.

When $\lambda=\frac12$, equations \eqref{qidian}, \eqref{fenmu case0-0}  and \eqref{etabiaosh} yield that
   \begin{align*}
     |\log(T-t)|(\alpha-\hat{\alpha}_t)= \frac{\int_{0}^t (T-u)^{\bar{\alpha}-1}\bar{\omega}_u\dif \zeta_u }{\frac{1}{|\log(T-t)|}\int_0^t (T-u)^{2\lambda-2}|\omega_u|^2 \dif  u}\xrightarrow{d} \frac{F_{\infty}+ {B(2H-1,\bar{\alpha})} H T^{2H-1} }{|\omega_T|^2}.
   \end{align*}

When $\lambda\in (0,1-H)$, equations \eqref{qidian} and \eqref{etabiaosh} imply that 
 \begin{align}
    & (T-t)^{\lambda-H}(\alpha-\hat{\alpha}_t)\notag\\
				&=\frac{(T-t)^{1-H-\lambda}\eta_t}{(T-t)^{1-2\lambda} \int_0^t |\omega _s|^2 (T-s)^{2\lambda-2}\dif s}\notag\\
                &=\frac{(T-t)^{1-H-\lambda}\Big[G_t +C_H \int_0^t  (T-u)^{\overline{\alpha}-1}\dif u \int_0^u (T-v)^{-\overline{\alpha}} (u-v)^{2H-2}\dif v\Big]}{(T-t)^{1-2\lambda} \int_0^t |\omega_s|^2 (T-s)^{2\lambda-2}\dif s},\label{eq.61-n}
 \end{align} where $C_H=H(2H-1)$. We introduce the stochastic integral $A_t:=\int_0^t (T-u)^{\bar{\alpha}-1}\dif \zeta_t$ and the function
\begin{align}\label{phituv}
    \phi_t(u,v)=(T-u)^{\overline{\alpha}-1}(T-v)^{-\overline{\alpha}}\mathbf{1}_{\{0\leq u\leq v\leq t\}}.
\end{align} It follows from the product formula for complex multiple Wiener-It\^o integral and the linearity of complex multiple Wiener-It\^o integral that 
 \begin{align}\label{zhj jg}
     A_t\cdot \overline{\omega}_t= G_t+ I_{1,1}(\phi_t)+C_H\int_{(0,t)^2}(T-u)^{\overline{\alpha}-1}(T-v)^{-\overline{\alpha}} \abs{u-v}^{2H-2}\dif u\dif v.
 \end{align}   Substituting equation \eqref{zhj jg} into equation \eqref{eq.61-n}, we have 
 \begin{align}
    & (T-t)^{\lambda-H}(\alpha-\hat{\alpha}_t)\notag\\
    &=\frac{(T-t)^{1-H-\lambda}\Big[A_t\cdot \overline{\omega}_t-I_{1,1}(\phi_t)-C_H \int_0^t (T-v)^{-\overline{\alpha}} \dif v \int_0^v  (T-u)^{\overline{\alpha}-1}(v-u)^{2H-2} \dif u \Big]}{(T-t)^{1-2\lambda} \int_0^t |\omega_s|^2 (T-s)^{2\lambda-2}\dif s}.\label{guodu dengshi}
 \end{align}
By Lemma~\ref{cal.result1}, we have
\begin{align*}
  \lim_{t \to T} \int_0^t   (T-v)^{-\overline{\alpha}} \dif v \int_0^v  (T-u)^{\overline{\alpha}-1}(v-u)^{2H-2} \dif u = \frac{B(2H-1, 1-\bar{\alpha})}{2H-1}T^{2H-1}.  
\end{align*}
Hence, when $\lambda\in (0,1-H)$,
 \begin{align*}
  \lim_{t\to T} (T-t)^{1-H-\lambda}  \abs{\int_0^t   (T-v)^{-\overline{\alpha}} \dif v \int_0^v  (T-u)^{\overline{\alpha}-1}(v-u)^{2H-2} \dif u} =0.
 \end{align*} Next, Proposition~\ref{last estimeate} implies that 
\begin{align*}
    \limsup_{t\to T} \E\big[\abs{I_{1,1}({\phi}_t)}^2\big]<\infty,
\end{align*}Since $\lambda\in (0,1-H)$, we have
\begin{align}\label{i11psit jixian}
     (T-t)^{1-H-\lambda}I_{1,1}(\phi_t)\xrightarrow{d} 0,\quad \text{as\,\,\,} t\to T.
 \end{align}
Proposition~\ref{prop Y} implies that the complex Gaussian random variable $\tilde{A}_T:=\lim_{t\to T}(T-t)^{1-H-\lambda} A_t $ exists in $L^2$ and 
\begin{align} \label{At jixian}
\E\big[\tilde{A}_T^2\big]=\frac{H\Gamma(2H)}{1-\lambda-H} \RE\left(\frac{\Gamma(2-{\alpha}-2H)}{\Gamma(1-{\alpha})}\right).
\end{align} Moreover, the two complex Gaussian random variables $\tilde{A}_T$ and $\omega_T$ are independent. Hence, we have
\begin{align}\label{ATjixian zuih}
     (T-t)^{1-H-\lambda} A_t\cdot \overline{\omega}_t \xrightarrow{d} \tilde{A}_T \cdot \bar{\omega}_T,\quad \text{as\,\,\,} t\to T.
\end{align}
Finally, by the results of \eqref{fenmu case0}, \eqref{guodu dengshi}-\eqref{i11psit jixian}, \eqref{ATjixian zuih} and Slutsky's lemma, we can compute for $\lambda\in (0,1-H)$, 
\begin{align}
     (T-t)^{\lambda-H}(\alpha-\hat{\alpha}_t)\xrightarrow{d} (1-2\lambda)\frac{\tilde{A}_T }{\omega_T}
\end{align} 
The ratio $\frac{\tilde{A}_T }{\omega_T}$ follows Cauchy distribution, i.e., $(1-2\lambda)\frac{\tilde{A}_T }{\omega_T} \sim (1-2\lambda) CR\left(\frac{\sigma^2_x}{\sigma^2_y}\right)$. The exact value of the scale parameter $\frac{\sigma^2_x}{\sigma^2_y}$ is obtained from equations \eqref{wt2 norm} and \eqref{At jixian}.

Similarly, we can obtain the result for the case of $\lambda=1-H$. This concludes the proof.
{\hfill\large{$\Box$}}  

\section*{Appendix}
In this section, we will present and prove some technical results that have been used in the proofs of the theorems.
\begin{lemma}\label{cal.result1}
  Assume $\RE(\alpha)\in(0, 1)$ and $H>\frac{1}{2}$. We have
  $$\lim_{t\to T}  \int_0^t  (T-u)^{\overline{\alpha}-1}\dif u \int_0^u (T-v)^{-\overline{\alpha}} (u-v)^{2H-2}\dif v = \frac{B(2H-1,\bar{\alpha})}{2H-1}T^{2H-1},$$
  where $B(a, b)=\int_0^1 (1-t)^{a-1}t^{b-1} \dif t$ for $Re(a)>0, Re(b)>0$ is the complex Beta function.
\end{lemma}

\begin{proof}
    Making the change of variables $x=T-v,\,y=u-v$ implies that for any $ \RE(\alpha)\in(0, 1)$,
\begin{align}
 & \lim_{t\to T}  \int_0^t  (T-u)^{\overline{\alpha}-1}\dif u \int_0^u (T-v)^{-\overline{\alpha}} (u-v)^{2H-2}\dif v \notag\\
    &=\int_0^T  (T-u)^{\overline{\alpha}-1}\dif u \int_0^u (T-v)^{-\overline{\alpha}} (u-v)^{2H-2}\dif v\notag\\
    &=\int_0^T x^{-\bar{\alpha}}\dif x\int_0^x (x-y)^{\overline{\alpha}-1} y^{2H-2}\dif y =\frac{B(2H-1,\bar{\alpha})}{2H-1}T^{2H-1}. 
\end{align}
\end{proof}

\begin{proposition}\label{last estimeate}
    Let $\psi_t(\cdot,\cdot)$ and $\phi_t(\cdot,\cdot)$ be given by \eqref{hehanshu01} and \eqref{phituv}, respectively. Then we have:
    \begin{itemize}
        \item [(i)]  If $\RE{\alpha}\in (1-H,1)$, then
   \begin{align}
    \lim_{t\to T} \E\big[\abs{I_{1,1}({\psi}_t)}^2\big]<\infty,\label{limsup last.02}
\end{align}
 \item [(ii)] If   $\RE{\alpha}\in (0,H)$, then 
    \begin{align}
    \lim_{t\to T} \E\big[\abs{I_{1,1}({\phi}_t)}^2\big]<\infty,\label{limsup last.01}
\end{align}
    \end{itemize}
  
\end{proposition}
 \begin{proof}
 It follows from equations \eqref{chongjifengongegx}-\eqref{duiouguanxi} that $\overline{I_{1,1}({\phi}_t)}=I_{1,1}(g_t)$, where 
 \begin{align*}
     g_t(u,v)=\overline{\phi_t(v,u)}&=(T-u)^{-{\alpha}}(T-v)^{{\alpha}-1}\mathbf{1}_{\{0\leq v\leq u\leq t\}}\\
     &=(T-u)^{\bar{\gamma}-1}(T-v)^{-\bar{\gamma}}\mathbf{1}_{\{0\leq v\leq u\leq t\}}, 
 \end{align*}  
 by setting $\bar{\gamma}=1-{\alpha}$. Therefore, by the definition of the function $\psi_t(\cdot, \cdot)$, statement (ii) is equivalent to statement (i). Thus, it suffices to prove (ii).
 
 Denote the constants $C_H=H(2H-1)$ and $\beta=2H-2$, and the integration regions $\Delta_0:=\set{0\le u_1\le v_1\le t, 0\le u_2\le v_2\le t}$ and $\Delta_0':=\set{0\le u_1\le v_1\le v_2, 0\le u_2\le v_2\le t}$.
    By It\^o's isometry and symmetry, we have 
    \begin{align}
      & \E\big[\abs{I_{1,1}({\phi}_t)}^2\big]\notag\\
      &=C_H^2\int_{\Delta_0}(T-u_1)^{\overline{\alpha}-1}(T-v_1)^{-\overline{\alpha}} (T-u_2)^{ {\alpha}-1}(T-v_2)^{- {\alpha}} \abs{u_1-u_2}^{2H-2}\abs{v_1-v_2}^{2H-2}\dif \vec{u}\dif \vec{v}\notag\\ 
      &=2C_H^2\RE\Bigg[\int_{\Delta_0'}(T-u_1)^{\overline{\alpha}-1}(T-v_1)^{-\overline{\alpha}} (T-u_2)^{ {\alpha}-1}(T-v_2)^{- {\alpha}} \abs{u_1-u_2}^{\beta}\abs{v_1-v_2}^{\beta}\dif \vec{u}\dif \vec{v}\Bigg].\label{zuihou1}
    \end{align} We divide the domain of integration $\Delta_0'$ into three subdomains:
    \begin{align*}
        \Delta_1:=&\set{0\leq u_2\leq u_1\leq v_1\leq v_2\leq t},\quad \Delta_2:=\set{0\leq u_1\leq u_2\leq v_1\leq v_2\leq t},\\
        {\rm and} \ \Delta_3:=& \set{0\leq u_1\leq v_1\leq u_2\leq v_2\leq t}.
    \end{align*}
 Making the change of variables $x=T-v_2,y=T-v_1, z=T-u_1, w=T-u_2$, we have 
\begin{align}
     & \lim_{t \to T} \int_{\Delta_1}(T-u_1)^{\overline{\alpha}-1}(T-v_1)^{-\overline{\alpha}} (T-u_2)^{ {\alpha}-1}(T-v_2)^{- {\alpha}} \abs{u_1-u_2}^{\beta}\abs{v_1-v_2}^{\beta}\dif \vec{u}\dif \vec{v}\notag\\
    = \, & \lim_{t \to T} \int_0^T w^{\alpha-1} \dif w \int_0^w  z^{\bar{\alpha}-1}(w-z)^{2H-2}\dif z\int_0^z y^{-\bar{\alpha} } \dif y \int_0^y x^{-{\alpha}} (y-x)^{2H-2}\dif x, \notag\\
    = \, & \frac{B(1-{\alpha},2H-1)B(2H-\alpha, 2H-1)}{4(2H-1)(H-\lambda )} T^{4H-2},\label{llas.050-2}
\end{align}
where $B(\cdot, \cdot)$ is the complex Beta function. Making the change of variables $x=T-v_2,y=T-v_1, z=T-u_2, w=T-u_1$, we have
\begin{align}
    & \lim_{t \to T} \int_{\Delta_2}(T-u_1)^{\overline{\alpha}-1}(T-v_1)^{-\overline{\alpha}} (T-u_2)^{ {\alpha}-1}(T-v_2)^{- {\alpha}} \abs{u_1-u_2}^{\beta}\abs{v_1-v_2}^{\beta}\dif \vec{u}\dif \vec{v}\notag\\
    = \, & \frac{B(1-{\alpha},2H-1)B(2H-\bar{\alpha} , 2H-1)}{4(2H-1)(H-\lambda)} T^{4H-2}.
\end{align}
Making the change of variables $x=T-v_2,y=T-u_2, z=T-v_1, w=T-u_1$, we have
\begin{align}
    &\lim_{t \to T} \int_{\Delta_3}(T-u_1)^{\overline{\alpha}-1}(T-v_1)^{-\overline{\alpha}} (T-u_2)^{ {\alpha}-1}(T-v_2)^{- {\alpha}} \abs{u_1-u_2}^{\beta}\abs{v_1-v_2}^{\beta}\dif \vec{u}\dif \vec{v}\notag\notag\\
    = \, & \int_{0\le x\le y\le z\le w\le T} w^{\bar{\alpha}-1}z^{-\bar{\alpha}}y^{\alpha-1}x^{-\alpha}(w-y)^{2H-2}(z-x)^{2H-2}\dif x \dif y\dif z\dif w,
\end{align}where the absolute value of the above integral is bounded by 
\begin{align}
    &\int_{0\le x\le y\le z\le w\le T} w^{ {\lambda}-1}z^{- {\lambda}}y^{\lambda-1}x^{-\lambda}(w-z)^{2H-2}(y-x)^{2H-2} \dif x\dif z
    \notag\\&=\frac{B(1-\lambda,2H-1)B(2H-\lambda,2H-1)}{2 (2H-1)^2}T^{4H-2}.\label{llas.050}
\end{align}
Substituting the above results \eqref{llas.050-2}-\eqref{llas.050} into equation \eqref{zuihou1}, we conclude that the inequality \eqref{limsup last.01} holds when $\lambda\in (0,H)$.   
 \end{proof}
 
\noindent  {\bf Acknowledgments}:
Y. Chen is supported by NSFC (12461029) and the PhD Research Startup Fund of Baoshan University. Y. Li is supported by the Youth Project of Hunan Provincial Natural Science Foundation of China (2024JJ6417) and Project of Education Department of Hunan Province (24B0187).

\noindent
\bibliographystyle{plain}

\end{document}